\documentclass[11pt]{article}

\usepackage[a4paper,margin=1in]{geometry}
\usepackage{amsmath,amssymb,amsfonts,amsthm}
\usepackage{mathtools}
\usepackage{authblk}
\usepackage{enumitem}
\usepackage{hyperref}
\usepackage{cite}
\usepackage{mathrsfs}
\usepackage{bm}

\newtheorem{theorem}{Theorem}[section]
\newtheorem{lemma}[theorem]{Lemma}
\newtheorem{corollary}[theorem]{Corollary}

\newtheorem{remark}[theorem]{Remark}

\newcommand{\T}{\mathbb{T}}
\newcommand{\R}{\mathbb{R}}

\title{\textbf{Kernel-Type Fractional Extensions of Reverse Callebaut, Rogers H\"older and Cauchy Schwarz Inequalities on Time Scales}}

\author[1]{Nimai Sarkar}
\affil[1]{Department of Mathematics, School of Advanced Sciences, VIT-AP University, Inavolu, Amaravati, Andhra Pradesh 522241, India}
\affil[ ]{\texttt{nimai.s@vitap.ac.in}}

\author[2]{Gobinda Ghosh}
\affil[2]{Mathematics Division, School of Advanced Sciences and Languages, VIT Bhopal University, Kothrikalan, Sehore, Madhya Pradesh 466114, India}
\affil[ ]{\texttt{sagarghosh798@gmail.com}}
\date{}

\begin{document}

\maketitle

\begin{abstract}
In this paper, we establish kernel-type fractional extensions of reverse Callebaut, Rogers--H\"older and Cauchy--Schwarz inequalities on time scales. The proposed results are formulated by means of a nonnegative kernel operator associated with the diamond-$\alpha$ integral, which enables the inclusion of memory effects within the framework of dynamic inequalities. By choosing particular kernels and time scales, the obtained inequalities reduce to their continuous, discrete and quantum counterparts. Moreover, when the kernel is taken as the identity kernel, the results recover the corresponding classical diamond-$\alpha$ dynamic inequalities. Several consequences are derived as special cases, including fractional reverse Cauchy--Schwarz inequalities and weighted reverse Rogers--H\"older inequalities. The results provide a unified approach for studying reverse inequalities in continuous, discrete, quantum and fractional dynamic settings and may be applied to obtain a priori estimates for Volterra-type and delay dynamic equations.
\end{abstract}


\medskip


\section{Introduction}

The theory of time scales was introduced by Hilger \cite{Hilger1988} in order to provide a unified framework for continuous and discrete analysis. A time scale $\T$ is an arbitrary nonempty closed subset of the real numbers. Hence, by choosing $\T=\R$, one obtains the classical continuous calculus, while the choice $\T=\mathbb{Z}$ leads to the discrete calculus. Other important examples include the quantum time scale $\T=q^{\mathbb{N}_{0}}$, where $q>1$. The main advantage of time scale calculus is that it allows one to establish results in a general form and then recover continuous, discrete and quantum versions as particular cases. Standard references on dynamic equations and inequalities on time scales include \cite{BohnerPeterson2001,BohnerPeterson2003,Agarwal2014}.

Dynamic inequalities on time scales have attracted considerable attention because of their usefulness in qualitative analysis of dynamic equations. Inequalities of H\"older, Minkowski, Jensen, Gronwall, Opial, Hardy and Cauchy--Schwarz type play a fundamental role in the study of existence, uniqueness, boundedness, stability and asymptotic behaviour of solutions. Among these inequalities, reverse forms are especially important because they provide upper bounds for the deviation between two sides of a classical inequality. Such estimates are useful when one needs to measure the gap in H\"older-type or Cauchy--Schwarz-type inequalities.

The diamond-$\alpha$ calculus is an important refinement of time scale calculus. If a function is both delta and nabla differentiable, then its diamond-$\alpha$ derivative is defined as a convex combination of the delta and nabla derivatives. Similarly, the diamond-$\alpha$ integral combines the delta and nabla integrals. More precisely, for $0\leq \alpha \leq 1$, the diamond-$\alpha$ integral is given by
\[
\int_{a}^{b} h(s)\diamond_{\alpha}s
=
\alpha\int_{a}^{b}h(s)\Delta s
+
(1-\alpha)\int_{a}^{b}h(s)\nabla s,
\]
provided that the delta and nabla integrals exist. For $\alpha=1$, the diamond-$\alpha$ integral reduces to the delta integral, while for $\alpha=0$, it reduces to the nabla integral. Thus, the diamond-$\alpha$ framework provides a flexible bridge between forward and backward dynamic analysis.

Recently, reverse versions of classical inequalities on time scales have been investigated using refinements and reverses of Young's inequality. In particular, reverse Callebaut, Rogers--H\"older and Cauchy--Schwarz inequalities have been extended to the diamond-$\alpha$ setting. These results unify discrete, continuous and quantum cases and provide a systematic way to study classical reverse inequalities in the language of time scales. However, most existing reverse dynamic inequalities do not incorporate fractional memory or general kernel effects.

Fractional and kernel-type operators are natural tools for describing memory, hereditary behaviour and nonlocal interactions. In many mathematical models arising from physics, biology, engineering and control theory, the present state depends not only on the current value of the unknown function but also on its past history. Such behaviour is commonly represented by integral operators with kernels. Motivated by this observation, we introduce a kernel-type fractional diamond-$\alpha$ integral operator of the form
\[
\mathcal{I}_{K}^{\alpha}h(t)
=
\int_{t_{1}}^{t}K(t,s)h(s)\diamond_{\alpha}s,
\]
where $K(t,s)\geq 0$ is a suitable kernel. For example, in the continuous case one may take
\[
K(t,s)=\frac{(t-s)^{\mu-1}}{\Gamma(\mu)},\qquad 0<\mu<1,
\]
which corresponds to a Riemann--Liouville-type fractional kernel. In this way, the proposed framework includes both ordinary dynamic inequalities and fractional-type dynamic inequalities as special cases.

The purpose of this paper is to establish kernel-type fractional extensions of reverse Callebaut, Rogers--H\"older and Cauchy--Schwarz inequalities on time scales. The main idea is to combine reverse Young-type inequalities with nonnegative kernel-weighted diamond-$\alpha$ integrals. Since the kernel is assumed to be nonnegative, the order of inequalities is preserved under integration. This allows us to derive fractional reverse inequalities in a unified dynamic setting.

The main contributions of this paper are summarized as follows:
\begin{enumerate}[label=(\roman*)]
    \item We formulate a nonnegative kernel-type diamond-$\alpha$ integral operator on time scales, which provides a unified framework for treating memory-dependent weighted dynamic inequalities.

    \item We establish kernel-type fractional reverse Callebaut inequalities on time scales by combining reverse Young-type estimates with the positivity of the proposed integral operator.

    \item We derive fractional reverse Rogers--H\"older inequalities involving conjugate indices $p$ and $q$ within the proposed kernel-type dynamic framework.

    \item We obtain reverse Cauchy--Schwarz inequalities as direct consequences of the fractional Rogers--H\"older inequalities by choosing $p=q=2$.

    \item We discuss several limiting and special cases, showing that the obtained results reduce to known diamond-$\alpha$ dynamic inequalities, as well as continuous, discrete and quantum forms, for suitable choices of the kernel and the underlying time scale.

    \item We indicate how the proposed inequalities can be used to derive a priori estimates for Volterra-type and delay dynamic equations on time scales.
\end{enumerate}

The results obtained in this work extend the scope of reverse dynamic inequalities from the classical diamond-$\alpha$ setting to a broader kernel-type fractional framework. In particular, if $K(t,s)=1$ and $t=t_{2}$, then the proposed inequalities reduce to the usual diamond-$\alpha$ reverse inequalities. If $\T=\R$, the results give continuous fractional integral inequalities, while if $\T=\mathbb{Z}$, they lead to discrete fractional sum inequalities. Similarly, choosing $\T=q^{\mathbb{N}_{0}}$ yields quantum-type fractional reverse inequalities.

The paper is organized as follows. In Section 2, we recall basic definitions and preliminary inequalities needed throughout the paper. In Section 3, we establish kernel-type reverse Callebaut inequalities on time scales. Section 4 is devoted to fractional reverse Rogers--H\"older inequalities. In Section 5, reverse Cauchy--Schwarz inequalities are obtained as consequences. Section 6 discusses special cases, including continuous, discrete and quantum forms. Finally, Section 7 presents applications to a priori estimates for Volterra-type and delay dynamic equations.

\section{Preliminaries}

In this section, we recall some basic notions and inequalities which will be
used throughout the paper. Let $\mathbb{T}$ be a time scale, that is, a
nonempty closed subset of $\mathbb{R}$. For $a,b\in\mathbb{T}$ with $a<b$, we
write
\[
    [a,b]_{\mathbb{T}}=[a,b]\cap\mathbb{T}.
\]
The basic theory of time scale calculus and dynamic inequalities may be found
in \cite{Hilger1988,BohnerPeterson2001,BohnerPeterson2003,Agarwal2014}.
Recent related developments on dynamic and fractional inequalities on time
scales can be found in \cite{RezkEtAl20XX,SarkarSenSahaHazarika2026,Sahir2025}.

Let $0\leq \alpha\leq 1$. If a function
$f:\mathbb{T}\to\mathbb{R}$ is both delta and nabla differentiable, then its
diamond-$\alpha$ derivative is defined by
\[
    f^{\diamond_{\alpha}}(t)
    =
    \alpha f^{\Delta}(t)+(1-\alpha)f^{\nabla}(t).
\]
Similarly, if $h:\mathbb{T}\to\mathbb{R}$ is delta and nabla integrable on
$[a,b]_{\mathbb{T}}$, then its diamond-$\alpha$ integral is defined by
\[
    \int_a^b h(s)\diamond_{\alpha}s
    =
    \alpha\int_a^b h(s)\Delta s
    +
    (1-\alpha)\int_a^b h(s)\nabla s.
\]
For $\alpha=1$, the above integral reduces to the delta integral, while for
$\alpha=0$, it reduces to the nabla integral.

We shall use the following refined reverse Young inequality due to Kittaneh
and Manasrah \cite{Kittaneh2010,Kittaneh2011}.

\begin{lemma}[Reverse Young inequality]
Let $\Phi,\Psi>0$ and $\delta\in[0,1]$. Then
\[
    \beta\left(\sqrt{\Phi}-\sqrt{\Psi}\right)^2
    \leq
    (1-\delta)\Phi+\delta\Psi-\Phi^{1-\delta}\Psi^{\delta}
    \leq
    \gamma\left(\sqrt{\Phi}-\sqrt{\Psi}\right)^2,
\]
where
\[
    \beta=\min\{1-\delta,\delta\},
    \qquad
    \gamma=\max\{1-\delta,\delta\}.
\]
\end{lemma}

The following bounded reverse form will also be used.

\begin{lemma}
Let $\Phi,\Psi\in[m,M]\subset(0,\infty)$ and $\delta\in[0,1]$. Then
\[
    (1-\delta)\Phi+\delta\Psi-\Phi^{1-\delta}\Psi^{\delta}
    \leq
    \gamma\left(\sqrt{M}-\sqrt{m}\right)^2,
\]
where
\[
    \gamma=\max\{1-\delta,\delta\}.
\]
\end{lemma}

For $h>0$, the Kantorovich ratio is defined by
\[
    K(h)=\frac{(h+1)^2}{4h}.
\]
It is well known that $K(h)\geq1$ and $K(h)=K(1/h)$. We shall use the following
Kantorovich-type reverse Young inequality.

\begin{lemma}
Let $\Phi,\Psi>0$, $\delta\in[0,1]$ and suppose that
\[
    0<L^{-1}\leq \frac{\Phi}{\Psi}\leq L<\infty,
    \qquad L>1.
\]
Then
\[
    (1-\delta)\Phi+\delta\Psi
    \leq
    K^{\gamma}(L)\Phi^{1-\delta}\Psi^{\delta},
\]
where
\[
    \gamma=\max\{1-\delta,\delta\}.
\]
\end{lemma}

Now we introduce the kernel-type diamond-$\alpha$ integral operator used in
this paper. Let
\[
    K:[a,b]_{\mathbb{T}}\times[a,b]_{\mathbb{T}}\to[0,\infty)
\]
be a nonnegative kernel. For a suitable function
$h:[a,b]_{\mathbb{T}}\to\mathbb{R}$, define
\[
    \mathcal{I}_{K}^{\alpha}h(t)
    =
    \int_a^t K(t,s)h(s)\diamond_{\alpha}s,
    \qquad t\in[a,b]_{\mathbb{T}},
\]
provided that the integral exists and is finite.

Throughout the paper, unless otherwise stated, we assume that
$w,f,g\in C([a,b]_{\mathbb{T}},\mathbb{R})$ are diamond-$\alpha$ integrable
functions, $K(t,s)\geq0$ for $a\leq s\leq t\leq b$, and all integrals appearing
in the sequel exist and are finite. When quotient-type estimates are considered,
we also assume that there exist constants $m,M>0$ such that
\[
    0<m\leq \frac{|f(s)|}{|g(s)|}\leq M<\infty,
    \qquad s\in[a,b]_{\mathbb{T}}.
\]

\section{Main Results}

Throughout this section, let $a,b\in\mathbb{T}$ with $a<b$ and let
$t\in(a,b]_{\mathbb{T}}$ be fixed. For simplicity of notation, we put
\[
    \Omega_t(s)=K(t,s)|w(s)|,\qquad s\in[a,t]_{\mathbb{T}},
\]
where $K(t,s)\geq0$. Thus,
\[
    \mathcal{I}_{K}^{\alpha}h(t)
    =
    \int_a^t K(t,s)h(s)\diamond_{\alpha}s.
\]
All functions and integrals appearing below are assumed to be well defined and
finite.

\begin{theorem}[Kernel-type reverse Callebaut inequality]
Let $w,f,g\in C([a,b]_{\mathbb{T}},\mathbb{R})$ be diamond-$\alpha$ integrable
functions and let $K(t,s)\geq0$. If $\delta\in[0,1]$, then
\[
\begin{aligned}
&2\beta\bigg[
\mathcal{I}_{K}^{\alpha}(|w||f|^2)(t)
\mathcal{I}_{K}^{\alpha}(|w||g|^2)(t)
-
\left(\mathcal{I}_{K}^{\alpha}(|w||f||g|)(t)\right)^2
\bigg]   \\
&\leq
\mathcal{I}_{K}^{\alpha}(|w||f|^2)(t)
\mathcal{I}_{K}^{\alpha}(|w||g|^2)(t)                                      \\
&\quad -
\mathcal{I}_{K}^{\alpha}
\left(|w||f|^{2(1-\delta)}|g|^{2\delta}\right)(t)
\mathcal{I}_{K}^{\alpha}
\left(|w||f|^{2\delta}|g|^{2(1-\delta)}\right)(t)                            \\
&\leq
2\gamma\bigg[
\mathcal{I}_{K}^{\alpha}(|w||f|^2)(t)
\mathcal{I}_{K}^{\alpha}(|w||g|^2)(t)
-
\left(\mathcal{I}_{K}^{\alpha}(|w||f||g|)(t)\right)^2
\bigg],
\end{aligned}
\]
where
\[
    \beta=\min\{1-\delta,\delta\},
    \qquad
    \gamma=\max\{1-\delta,\delta\}.
\]
\end{theorem}

\begin{proof}
Let $s,r\in[a,t]_{\mathbb{T}}$. Applying the reverse Young inequality to
\[
    \Phi=|f(s)|^2|g(r)|^2,
    \qquad
    \Psi=|g(s)|^2|f(r)|^2,
\]
we obtain
\[
\begin{aligned}
&\beta\left(|f(s)||g(r)|-|g(s)||f(r)|\right)^2                                      \\
&\leq
(1-\delta)|f(s)|^2|g(r)|^2
+\delta |g(s)|^2|f(r)|^2                                                            \\
&\quad
-|f(s)|^{2(1-\delta)}|g(s)|^{2\delta}
 |f(r)|^{2\delta}|g(r)|^{2(1-\delta)}                                               \\
&\leq
\gamma\left(|f(s)||g(r)|-|g(s)||f(r)|\right)^2 .
\end{aligned}
\]
Since $\Omega_t(s)\geq0$ and $\Omega_t(r)\geq0$, multiplying the above
inequality by $\Omega_t(s)\Omega_t(r)$ preserves the order. Therefore,
\[
\begin{aligned}
&\beta\Omega_t(s)\Omega_t(r)
\left(|f(s)||g(r)|-|g(s)||f(r)|\right)^2                                      \\
&\leq
\Omega_t(s)\Omega_t(r)
\bigg[
(1-\delta)|f(s)|^2|g(r)|^2
+\delta |g(s)|^2|f(r)|^2                                                        \\
&\quad
-|f(s)|^{2(1-\delta)}|g(s)|^{2\delta}
 |f(r)|^{2\delta}|g(r)|^{2(1-\delta)}
\bigg]                                                                         \\
&\leq
\gamma\Omega_t(s)\Omega_t(r)
\left(|f(s)||g(r)|-|g(s)||f(r)|\right)^2 .
\end{aligned}
\]
Now integrate the above inequality first with respect to $s$ and then with
respect to $r$ over $[a,t]_{\mathbb{T}}$ using the diamond-$\alpha$ integral.

The square term gives
\[
\begin{aligned}
&\int_a^t\int_a^t
\Omega_t(s)\Omega_t(r)
\left(|f(s)||g(r)|-|g(s)||f(r)|\right)^2
\diamond_{\alpha}s\diamond_{\alpha}r                                      \\
&=
2\bigg[
\mathcal{I}_{K}^{\alpha}(|w||f|^2)(t)
\mathcal{I}_{K}^{\alpha}(|w||g|^2)(t)
-
\left(\mathcal{I}_{K}^{\alpha}(|w||f||g|)(t)\right)^2
\bigg].
\end{aligned}
\]
Moreover, the middle term becomes
\[
\begin{aligned}
&
\mathcal{I}_{K}^{\alpha}(|w||f|^2)(t)
\mathcal{I}_{K}^{\alpha}(|w||g|^2)(t)                                      \\
&\quad -
\mathcal{I}_{K}^{\alpha}
\left(|w||f|^{2(1-\delta)}|g|^{2\delta}\right)(t)
\mathcal{I}_{K}^{\alpha}
\left(|w||f|^{2\delta}|g|^{2(1-\delta)}\right)(t).
\end{aligned}
\]
Combining these identities, we arrive at the desired inequality.
\end{proof}

\begin{remark}
If $K(t,s)=1$ and $t=b$, then the above theorem reduces to the usual
diamond-$\alpha$ reverse Callebaut inequality on time scales. Thus, the result
is a genuine kernel-type extension of the known dynamic inequality.
\end{remark}

\begin{remark}
The nonnegativity assumption $K(t,s)\geq0$ is essential. It guarantees that the
pointwise reverse Young inequality can be integrated without changing the
direction of the inequality.
\end{remark}

\begin{theorem}[Bounded kernel-type reverse Callebaut inequality]
Assume that the hypotheses of the previous theorem hold. Suppose further that
there exist constants $m,M>0$ such that
\[
    0<m\leq \frac{|f(s)|}{|g(s)|}\leq M<\infty,
    \qquad s\in[a,t]_{\mathbb{T}}.
\]
Then, for $\delta\in[0,1]$,
\[
\begin{aligned}
&
\mathcal{I}_{K}^{\alpha}(|w||f|^2)(t)
\mathcal{I}_{K}^{\alpha}(|w||g|^2)(t)                                      \\
&\quad -
\mathcal{I}_{K}^{\alpha}
\left(|w||f|^{2(1-\delta)}|g|^{2\delta}\right)(t)
\mathcal{I}_{K}^{\alpha}
\left(|w||f|^{2\delta}|g|^{2(1-\delta)}\right)(t)                            \\
&\leq
\gamma(M-m)^2
\left[
\mathcal{I}_{K}^{\alpha}(|w||g|^2)(t)
\right]^2,
\end{aligned}
\]
where
\[
    \gamma=\max\{1-\delta,\delta\}.
\]
\end{theorem}

\begin{proof}
For $s,r\in[a,t]_{\mathbb{T}}$, define
\[
    \Phi=\frac{|f(s)|^2}{|g(s)|^2},
    \qquad
    \Psi=\frac{|f(r)|^2}{|g(r)|^2}.
\]
By the assumed quotient condition,
\[
    m^2\leq \Phi,\Psi\leq M^2.
\]
Using the bounded reverse Young inequality, we obtain
\[
\begin{aligned}
&(1-\delta)\frac{|f(s)|^2}{|g(s)|^2}
+\delta\frac{|f(r)|^2}{|g(r)|^2}
-\left(\frac{|f(s)|^2}{|g(s)|^2}\right)^{1-\delta}
 \left(\frac{|f(r)|^2}{|g(r)|^2}\right)^{\delta}       \\
&\leq
\gamma(M-m)^2.
\end{aligned}
\]
Multiplying both sides by $|g(s)|^2|g(r)|^2$, we get
\[
\begin{aligned}
&(1-\delta)|f(s)|^2|g(r)|^2
+\delta |g(s)|^2|f(r)|^2                                      \\
&\quad
-|f(s)|^{2(1-\delta)}|g(s)|^{2\delta}
 |f(r)|^{2\delta}|g(r)|^{2(1-\delta)}                         \\
&\leq
\gamma(M-m)^2|g(s)|^2|g(r)|^2.
\end{aligned}
\]
Now multiply by $\Omega_t(s)\Omega_t(r)$ and integrate with respect to $s$ and
$r$ over $[a,t]_{\mathbb{T}}$. The left-hand side becomes
\[
\begin{aligned}
&
\mathcal{I}_{K}^{\alpha}(|w||f|^2)(t)
\mathcal{I}_{K}^{\alpha}(|w||g|^2)(t)                                      \\
&\quad -
\mathcal{I}_{K}^{\alpha}
\left(|w||f|^{2(1-\delta)}|g|^{2\delta}\right)(t)
\mathcal{I}_{K}^{\alpha}
\left(|w||f|^{2\delta}|g|^{2(1-\delta)}\right)(t),
\end{aligned}
\]
while the right-hand side becomes
\[
    \gamma(M-m)^2
    \left[
    \mathcal{I}_{K}^{\alpha}(|w||g|^2)(t)
    \right]^2.
\]
Hence the required inequality follows.
\end{proof}

\begin{remark}
Taking $\delta=\frac12$ in the previous theorem gives a kernel-type reverse
Cauchy--Schwarz inequality.
\end{remark}

\begin{corollary}[Kernel-type reverse Cauchy--Schwarz inequality]
Under the hypotheses of Theorem 3.2, one has
\[
\begin{aligned}
&
\mathcal{I}_{K}^{\alpha}(|w||f|^2)(t)
\mathcal{I}_{K}^{\alpha}(|w||g|^2)(t)
-
\left[
\mathcal{I}_{K}^{\alpha}(|w||f||g|)(t)
\right]^2                                                    \\
&\leq
\frac12(M-m)^2
\left[
\mathcal{I}_{K}^{\alpha}(|w||g|^2)(t)
\right]^2.
\end{aligned}
\]
\end{corollary}

\begin{theorem}[Kernel-type reverse Rogers--H\"older inequality]
Let $p,q>1$ with
\[
    \frac1p+\frac1q=1.
\]
Assume that
\[
    A(t):=\mathcal{I}_{K}^{\alpha}(|w||f|^p)(t)>0,
    \qquad
    B(t):=\mathcal{I}_{K}^{\alpha}(|w||g|^q)(t)>0.
\]
Then
\[
\begin{aligned}
&2\beta
\left[
\sqrt{A(t)B(t)}
-
\mathcal{I}_{K}^{\alpha}
\left(|w||f|^{p/2}|g|^{q/2}\right)(t)
\right]
A(t)^{\frac1p-\frac12}
B(t)^{\frac1q-\frac12}                                      \\
&\leq
A(t)^{1/p}B(t)^{1/q}
-
\mathcal{I}_{K}^{\alpha}(|w||f||g|)(t)                       \\
&\leq
2\gamma
\left[
\sqrt{A(t)B(t)}
-
\mathcal{I}_{K}^{\alpha}
\left(|w||f|^{p/2}|g|^{q/2}\right)(t)
\right]
A(t)^{\frac1p-\frac12}
B(t)^{\frac1q-\frac12},
\end{aligned}
\]
where
\[
    \beta=\min\left\{\frac1p,\frac1q\right\},
    \qquad
    \gamma=\max\left\{\frac1p,\frac1q\right\}.
\]
\end{theorem}

\begin{proof}
Put
\[
    A(t)=\mathcal{I}_{K}^{\alpha}(|w||f|^p)(t),
    \qquad
    B(t)=\mathcal{I}_{K}^{\alpha}(|w||g|^q)(t).
\]
For $s\in[a,t]_{\mathbb{T}}$, define
\[
    \Phi(s)=\frac{\Omega_t(s)|f(s)|^p}{A(t)},
    \qquad
    \Psi(s)=\frac{\Omega_t(s)|g(s)|^q}{B(t)}.
\]
Since $A(t)>0$ and $B(t)>0$, the above quantities are well defined and
nonnegative. Applying the refined reverse Young inequality with
\[
    \delta=\frac1q,
    \qquad
    1-\delta=\frac1p,
\]
we obtain
\[
\begin{aligned}
&\beta\left(\sqrt{\Phi(s)}-\sqrt{\Psi(s)}\right)^2     \\
&\leq
\frac1p\Phi(s)+\frac1q\Psi(s)
-\Phi(s)^{1/p}\Psi(s)^{1/q}                            \\
&\leq
\gamma\left(\sqrt{\Phi(s)}-\sqrt{\Psi(s)}\right)^2 .
\end{aligned}
\]
Integrating over $[a,t]_{\mathbb{T}}$, we get
\[
\begin{aligned}
&\beta\int_a^t
\left(\sqrt{\Phi(s)}-\sqrt{\Psi(s)}\right)^2
\diamond_{\alpha}s                                      \\
&\leq
\int_a^t
\left[
\frac1p\Phi(s)+\frac1q\Psi(s)
-\Phi(s)^{1/p}\Psi(s)^{1/q}
\right]\diamond_{\alpha}s                               \\
&\leq
\gamma\int_a^t
\left(\sqrt{\Phi(s)}-\sqrt{\Psi(s)}\right)^2
\diamond_{\alpha}s .
\end{aligned}
\]
Now,
\[
    \int_a^t \Phi(s)\diamond_{\alpha}s=1,
    \qquad
    \int_a^t \Psi(s)\diamond_{\alpha}s=1.
\]
Therefore,
\[
\begin{aligned}
\int_a^t
\left(\sqrt{\Phi(s)}-\sqrt{\Psi(s)}\right)^2
\diamond_{\alpha}s
&=
2-2\int_a^t\sqrt{\Phi(s)\Psi(s)}\diamond_{\alpha}s       \\
&=
2\left[
1-
\frac{
\mathcal{I}_{K}^{\alpha}
\left(|w||f|^{p/2}|g|^{q/2}\right)(t)
}
{\sqrt{A(t)B(t)}}
\right].
\end{aligned}
\]
Also,
\[
\begin{aligned}
&\int_a^t
\left[
\frac1p\Phi(s)+\frac1q\Psi(s)
-\Phi(s)^{1/p}\Psi(s)^{1/q}
\right]\diamond_{\alpha}s                                \\
&=
1-
\frac{
\mathcal{I}_{K}^{\alpha}(|w||f||g|)(t)
}
{
A(t)^{1/p}B(t)^{1/q}
}.
\end{aligned}
\]
Hence,
\[
\begin{aligned}
&2\beta
\left[
1-
\frac{
\mathcal{I}_{K}^{\alpha}
\left(|w||f|^{p/2}|g|^{q/2}\right)(t)
}
{\sqrt{A(t)B(t)}}
\right]                                                   \\
&\leq
1-
\frac{
\mathcal{I}_{K}^{\alpha}(|w||f||g|)(t)
}
{
A(t)^{1/p}B(t)^{1/q}
}                                                         \\
&\leq
2\gamma
\left[
1-
\frac{
\mathcal{I}_{K}^{\alpha}
\left(|w||f|^{p/2}|g|^{q/2}\right)(t)
}
{\sqrt{A(t)B(t)}}
\right].
\end{aligned}
\]
Multiplying throughout by $A(t)^{1/p}B(t)^{1/q}$ gives
\[
\begin{aligned}
&2\beta
\left[
\sqrt{A(t)B(t)}
-
\mathcal{I}_{K}^{\alpha}
\left(|w||f|^{p/2}|g|^{q/2}\right)(t)
\right]
A(t)^{\frac1p-\frac12}
B(t)^{\frac1q-\frac12}                                      \\
&\leq
A(t)^{1/p}B(t)^{1/q}
-
\mathcal{I}_{K}^{\alpha}(|w||f||g|)(t)                       \\
&\leq
2\gamma
\left[
\sqrt{A(t)B(t)}
-
\mathcal{I}_{K}^{\alpha}
\left(|w||f|^{p/2}|g|^{q/2}\right)(t)
\right]
A(t)^{\frac1p-\frac12}
B(t)^{\frac1q-\frac12}.
\end{aligned}
\]
This completes the proof.
\end{proof}

\begin{remark}
For $p=q=2$, the above result reduces to the reverse Cauchy--Schwarz gap in the
kernel-type diamond-$\alpha$ framework.
\end{remark}

\begin{remark}
If $K(t,s)=1$, then the theorem gives the standard diamond-$\alpha$ reverse
Rogers--H\"older inequality. If, in addition, $\mathbb{T}=\mathbb{R}$ or
$\mathbb{T}=\mathbb{Z}$, the continuous or discrete versions are recovered,
respectively.
\end{remark}

\begin{theorem}[Bounded reverse Rogers--H\"older inequality]
Let $p,q>1$ with $1/p+1/q=1$. Assume that
\[
    \mathcal{I}_{K}^{\alpha}(|w|)(t)=1.
\]
Suppose further that there exist constants $m,M,n,N>0$ such that
\[
    0<m\leq |f(s)|\leq M<\infty,
    \qquad
    0<n\leq |g(s)|\leq N<\infty,
\]
for all $s\in[a,t]_{\mathbb{T}}$. Then
\[
\begin{aligned}
0
&\leq
\left[
\mathcal{I}_{K}^{\alpha}(|w||f|^p)(t)
\right]^{1/p}
\left[
\mathcal{I}_{K}^{\alpha}(|w||g|^q)(t)
\right]^{1/q}
-
\mathcal{I}_{K}^{\alpha}(|w||f||g|)(t)                         \\
&\leq
\gamma \Lambda^2
\left[
\mathcal{I}_{K}^{\alpha}(|w||f|^p)(t)
\right]^{1/p}
\left[
\mathcal{I}_{K}^{\alpha}(|w||g|^q)(t)
\right]^{1/q},
\end{aligned}
\]
where
\[
    \gamma=\max\left\{\frac1p,\frac1q\right\}
\]
and
\[
    \Lambda=
    \max\left\{
    \left(\frac{M}{m}\right)^{p/2},
    \left(\frac{N}{n}\right)^{q/2}
    \right\}
    -
    \min\left\{
    \left(\frac{m}{M}\right)^{p/2},
    \left(\frac{n}{N}\right)^{q/2}
    \right\}.
\]
\end{theorem}

\begin{proof}
Let
\[
    A(t)=\mathcal{I}_{K}^{\alpha}(|w||f|^p)(t),
    \qquad
    B(t)=\mathcal{I}_{K}^{\alpha}(|w||g|^q)(t).
\]
Since $\mathcal{I}_{K}^{\alpha}(|w|)(t)=1$ and
$m\leq |f(s)|\leq M$, we have
\[
    m^p\leq A(t)\leq M^p.
\]
Similarly,
\[
    n^q\leq B(t)\leq N^q.
\]
Thus,
\[
    \left(\frac{m}{M}\right)^p
    \leq
    \frac{|f(s)|^p}{A(t)}
    \leq
    \left(\frac{M}{m}\right)^p,
\]
and
\[
    \left(\frac{n}{N}\right)^q
    \leq
    \frac{|g(s)|^q}{B(t)}
    \leq
    \left(\frac{N}{n}\right)^q.
\]
Put
\[
    \Phi(s)=\frac{|f(s)|^p}{A(t)},
    \qquad
    \Psi(s)=\frac{|g(s)|^q}{B(t)}.
\]
Then both $\Phi(s)$ and $\Psi(s)$ are bounded between the two positive
quantities
\[
    \min\left\{
    \left(\frac{m}{M}\right)^p,
    \left(\frac{n}{N}\right)^q
    \right\}
\]
and
\[
    \max\left\{
    \left(\frac{M}{m}\right)^p,
    \left(\frac{N}{n}\right)^q
    \right\}.
\]
Using the bounded reverse Young inequality with $\delta=1/q$, we get
\[
\begin{aligned}
&
\frac1p\Phi(s)+\frac1q\Psi(s)
-\Phi(s)^{1/p}\Psi(s)^{1/q}                                      \\
&\leq
\gamma
\left[
\max\left\{
\left(\frac{M}{m}\right)^{p/2},
\left(\frac{N}{n}\right)^{q/2}
\right\}
-
\min\left\{
\left(\frac{m}{M}\right)^{p/2},
\left(\frac{n}{N}\right)^{q/2}
\right\}
\right]^2                                                        \\
&=
\gamma\Lambda^2.
\end{aligned}
\]
Multiplying by $\Omega_t(s)$ and integrating over $[a,t]_{\mathbb{T}}$, we
obtain
\[
\begin{aligned}
&
\int_a^t \Omega_t(s)
\left[
\frac1p\Phi(s)+\frac1q\Psi(s)
-\Phi(s)^{1/p}\Psi(s)^{1/q}
\right]\diamond_{\alpha}s                                      \\
&\leq
\gamma\Lambda^2
\int_a^t \Omega_t(s)\diamond_{\alpha}s.
\end{aligned}
\]
Since $\mathcal{I}_{K}^{\alpha}(|w|)(t)=1$, the right-hand side becomes
$\gamma\Lambda^2$. On the other hand,
\[
\begin{aligned}
&\int_a^t \Omega_t(s)
\left[
\frac1p\Phi(s)+\frac1q\Psi(s)
-\Phi(s)^{1/p}\Psi(s)^{1/q}
\right]\diamond_{\alpha}s                                      \\
&=
1-
\frac{
\mathcal{I}_{K}^{\alpha}(|w||f||g|)(t)
}
{
A(t)^{1/p}B(t)^{1/q}
}.
\end{aligned}
\]
Therefore,
\[
    1-
\frac{
\mathcal{I}_{K}^{\alpha}(|w||f||g|)(t)
}
{
A(t)^{1/p}B(t)^{1/q}
}
\leq
\gamma\Lambda^2.
\]
Multiplying by $A(t)^{1/p}B(t)^{1/q}$ gives the required upper estimate.
The nonnegativity of the left-hand side follows from the usual H\"older
inequality in the kernel-type diamond-$\alpha$ setting.
\end{proof}

\begin{remark}
The normalization condition
\[
    \mathcal{I}_{K}^{\alpha}(|w|)(t)=1
\]
means that the kernel-weighted measure is normalized. If this condition is not
satisfied, one may normalize the weight by replacing $|w(s)|$ with
\[
    \frac{|w(s)|}
    {\mathcal{I}_{K}^{\alpha}(|w|)(t)}.
\]
\end{remark}

\begin{corollary}[Bounded reverse Cauchy--Schwarz inequality]
Assume that the hypotheses of Theorem 3.4 hold with $p=q=2$. Then
\[
\begin{aligned}
0
&\leq
\left[
\mathcal{I}_{K}^{\alpha}(|w||f|^2)(t)
\right]^{1/2}
\left[
\mathcal{I}_{K}^{\alpha}(|w||g|^2)(t)
\right]^{1/2}
-
\mathcal{I}_{K}^{\alpha}(|w||f||g|)(t)                         \\
&\leq
\frac12
\left[
\max\left\{
\frac{M}{m},\frac{N}{n}
\right\}
-
\min\left\{
\frac{m}{M},\frac{n}{N}
\right\}
\right]^2
\left[
\mathcal{I}_{K}^{\alpha}(|w||f|^2)(t)
\right]^{1/2}
\left[
\mathcal{I}_{K}^{\alpha}(|w||g|^2)(t)
\right]^{1/2}.
\end{aligned}
\]
\end{corollary}

\begin{theorem}[Kantorovich-type reverse Rogers--H\"older inequality]
Let $p,q>1$ with $1/p+1/q=1$ and assume that
\[
    \mathcal{I}_{K}^{\alpha}(|w|)(t)=1.
\]
Suppose that
\[
    0<m\leq |f(s)|\leq M<\infty,
    \qquad
    0<n\leq |g(s)|\leq N<\infty,
\]
for all $s\in[a,t]_{\mathbb{T}}$. Then
\[
\begin{aligned}
&
\left[
\mathcal{I}_{K}^{\alpha}(|w||f|^p)(t)
\right]^{1/p}
\left[
\mathcal{I}_{K}^{\alpha}(|w||g|^q)(t)
\right]^{1/q}                                                \\
&\leq
K^{\gamma}
\left[
\left(\frac{M}{m}\right)^p
\left(\frac{N}{n}\right)^q
\right]
\mathcal{I}_{K}^{\alpha}(|w||f||g|)(t),
\end{aligned}
\]
where
\[
    \gamma=\max\left\{\frac1p,\frac1q\right\}
\]
and
\[
    K(h)=\frac{(h+1)^2}{4h},\qquad h>0.
\]
\end{theorem}

\begin{proof}
Let
\[
    A(t)=\mathcal{I}_{K}^{\alpha}(|w||f|^p)(t),
    \qquad
    B(t)=\mathcal{I}_{K}^{\alpha}(|w||g|^q)(t).
\]
As in the proof of Theorem 3.4, we have
\[
    m^p\leq A(t)\leq M^p,
    \qquad
    n^q\leq B(t)\leq N^q.
\]
Define
\[
    \Phi(s)=\frac{|f(s)|^p}{A(t)},
    \qquad
    \Psi(s)=\frac{|g(s)|^q}{B(t)}.
\]
Then
\[
    \frac{\Phi(s)}{\Psi(s)}
    =
    \frac{|f(s)|^p B(t)}{|g(s)|^q A(t)}.
\]
Using the bounds on $f,g,A(t)$ and $B(t)$, we obtain
\[
\left[
\left(\frac{M}{m}\right)^p
\left(\frac{N}{n}\right)^q
\right]^{-1}
\leq
\frac{\Phi(s)}{\Psi(s)}
\leq
\left(\frac{M}{m}\right)^p
\left(\frac{N}{n}\right)^q.
\]
Let
\[
    L=
    \left(\frac{M}{m}\right)^p
    \left(\frac{N}{n}\right)^q.
\]
By the Kantorovich-type Young inequality with $\delta=1/q$, we have
\[
    \frac1p\Phi(s)+\frac1q\Psi(s)
    \leq
    K^{\gamma}(L)\Phi(s)^{1/p}\Psi(s)^{1/q}.
\]
Multiplying by $\Omega_t(s)$ and integrating over $[a,t]_{\mathbb{T}}$, we get
\[
\begin{aligned}
&\int_a^t\Omega_t(s)
\left[
\frac1p\Phi(s)+\frac1q\Psi(s)
\right]\diamond_{\alpha}s                                      \\
&\leq
K^{\gamma}(L)
\int_a^t\Omega_t(s)\Phi(s)^{1/p}\Psi(s)^{1/q}
\diamond_{\alpha}s.
\end{aligned}
\]
The left-hand side is equal to $1$, since
$\mathcal{I}_{K}^{\alpha}(|w|)(t)=1$. The right-hand side is
\[
    K^{\gamma}(L)
    \frac{
    \mathcal{I}_{K}^{\alpha}(|w||f||g|)(t)
    }
    {
    A(t)^{1/p}B(t)^{1/q}
    }.
\]
Therefore,
\[
    1
    \leq
    K^{\gamma}(L)
    \frac{
    \mathcal{I}_{K}^{\alpha}(|w||f||g|)(t)
    }
    {
    A(t)^{1/p}B(t)^{1/q}
    }.
\]
Rearranging gives
\[
    A(t)^{1/p}B(t)^{1/q}
    \leq
    K^{\gamma}(L)
    \mathcal{I}_{K}^{\alpha}(|w||f||g|)(t),
\]
which is the desired result.
\end{proof}

\begin{remark}
The previous theorem provides a multiplicative reverse form of the
Rogers--H\"older inequality. The constant is expressed in terms of the
Kantorovich ratio and depends only on the bounds of $f$ and $g$.
\end{remark}

\begin{remark}
All the above results reduce to the corresponding non-kernel diamond-$\alpha$
inequalities when $K(t,s)=1$. Moreover, by choosing
$\mathbb{T}=\mathbb{R}$, $\mathbb{T}=\mathbb{Z}$, or
$\mathbb{T}=q^{\mathbb{N}_0}$, one obtains continuous, discrete and quantum
forms, respectively.
\end{remark}

\section{Validation through Examples}

To verify the applicability of Theorem~3.1, we consider two representative
time scales, namely the continuous time scale $\mathbb{T}=\mathbb{R}$ and the
discrete time scale $\mathbb{T}=\mathbb{Z}$.

\subsection{Example 1.}

Let
\[
    \mathbb{T}=\mathbb{R}, \qquad a=0, \qquad t=1, \qquad \alpha=1.
\]
Then the diamond-$\alpha$ integral reduces to the usual integral. Choose
\[
    w(s)=1, \qquad f(s)=1+s, \qquad g(s)=1,
\]
and let
\[
    K(1,s)=1+s, \qquad 0\leq s\leq 1.
\]
For $\delta=1/4$, we have
\[
    \beta=\frac14, \qquad \gamma=\frac34.
\]

Now,
\[
\mathcal{I}_{K}^{\alpha}(|w||f|^2)(1)
=
\int_0^1 (1+s)^3\,ds
=
\frac{15}{4},
\]
\[
\mathcal{I}_{K}^{\alpha}(|w||g|^2)(1)
=
\int_0^1 (1+s)\,ds
=
\frac32,
\]
and
\[
\mathcal{I}_{K}^{\alpha}(|w||f||g|)(1)
=
\int_0^1 (1+s)^2\,ds
=
\frac73.
\]
Hence,
\[
\begin{aligned}
D
&=
\mathcal{I}_{K}^{\alpha}(|w||f|^2)(1)
\mathcal{I}_{K}^{\alpha}(|w||g|^2)(1)
-
\left[
\mathcal{I}_{K}^{\alpha}(|w||f||g|)(1)
\right]^2  \\
&=
\frac{15}{4}\cdot\frac32-\left(\frac73\right)^2
=
\frac{13}{72}.
\end{aligned}
\]

Furthermore, since
\[
    2(1-\delta)=\frac32,
    \qquad
    2\delta=\frac12,
\]
we obtain
\[
\begin{aligned}
\mathcal{I}_{K}^{\alpha}
\left(|w||f|^{2(1-\delta)}|g|^{2\delta}\right)(1)
&=
\int_0^1 (1+s)^{5/2}\,ds  \\
&=
\frac{16\sqrt{2}-2}{7},
\end{aligned}
\]
and
\[
\begin{aligned}
\mathcal{I}_{K}^{\alpha}
\left(|w||f|^{2\delta}|g|^{2(1-\delta)}\right)(1)
&=
\int_0^1 (1+s)^{3/2}\,ds  \\
&=
\frac{8\sqrt{2}-2}{5}.
\end{aligned}
\]
Therefore,
\[
\begin{aligned}
E
&=
\mathcal{I}_{K}^{\alpha}(|w||f|^2)(1)
\mathcal{I}_{K}^{\alpha}(|w||g|^2)(1)  \\
&\quad -
\mathcal{I}_{K}^{\alpha}
\left(|w||f|^{2(1-\delta)}|g|^{2\delta}\right)(1)
\mathcal{I}_{K}^{\alpha}
\left(|w||f|^{2\delta}|g|^{2(1-\delta)}\right)(1) \\
&=
\frac{384\sqrt{2}-505}{280}.
\end{aligned}
\]

Consequently,
\[
    2\beta D=\frac{13}{144}\approx 0.09027,
\]
\[
    E=\frac{384\sqrt{2}-505}{280}\approx 0.13592,
\]
and
\[
    2\gamma D=\frac{13}{48}\approx 0.27083.
\]
Thus,
\[
    2\beta D\leq E\leq 2\gamma D.
\]
This verifies Theorem~3.1 for the continuous time scale.

\subsection{Example 2.}

Let
\[
    \mathbb{T}=\mathbb{Z}, \qquad a=0, \qquad t=3, \qquad \alpha=1.
\]
Then the diamond-$\alpha$ integral reduces to the delta integral, and
\[
    \int_0^3 h(s)\Delta s=\sum_{s=0}^{2}h(s).
\]
Choose
\[
    w(s)=1, \qquad f(s)=1+s, \qquad g(s)=1,
\]
and let
\[
    K(3,s)=1+s, \qquad s=0,1,2.
\]
For $\delta=1/4$, we again have
\[
    \beta=\frac14, \qquad \gamma=\frac34.
\]

Now,
\[
\mathcal{I}_{K}^{\alpha}(|w||f|^2)(3)
=
\sum_{s=0}^{2}(1+s)^3
=
1^3+2^3+3^3
=
36,
\]
\[
\mathcal{I}_{K}^{\alpha}(|w||g|^2)(3)
=
\sum_{s=0}^{2}(1+s)
=
1+2+3
=
6,
\]
and
\[
\mathcal{I}_{K}^{\alpha}(|w||f||g|)(3)
=
\sum_{s=0}^{2}(1+s)^2
=
1^2+2^2+3^2
=
14.
\]
Hence,
\[
\begin{aligned}
D
&=
\mathcal{I}_{K}^{\alpha}(|w||f|^2)(3)
\mathcal{I}_{K}^{\alpha}(|w||g|^2)(3)
-
\left[
\mathcal{I}_{K}^{\alpha}(|w||f||g|)(3)
\right]^2  \\
&=
36\cdot 6-14^2
=
20.
\end{aligned}
\]

Furthermore,
\[
\begin{aligned}
\mathcal{I}_{K}^{\alpha}
\left(|w||f|^{2(1-\delta)}|g|^{2\delta}\right)(3)
&=
\sum_{s=0}^{2}(1+s)^{5/2} \\
&=
1+4\sqrt{2}+9\sqrt{3},
\end{aligned}
\]
and
\[
\begin{aligned}
\mathcal{I}_{K}^{\alpha}
\left(|w||f|^{2\delta}|g|^{2(1-\delta)}\right)(3)
&=
\sum_{s=0}^{2}(1+s)^{3/2} \\
&=
1+2\sqrt{2}+3\sqrt{3}.
\end{aligned}
\]
Therefore,
\[
\begin{aligned}
E
&=
36\cdot 6
-
(1+4\sqrt{2}+9\sqrt{3})(1+2\sqrt{2}+3\sqrt{3}) \\
&=
118-6\sqrt{2}-12\sqrt{3}-30\sqrt{6}.
\end{aligned}
\]

Consequently,
\[
    2\beta D=10,
\]
\[
    E=118-6\sqrt{2}-12\sqrt{3}-30\sqrt{6}
    \approx 15.2454,
\]
and
\[
    2\gamma D=30.
\]
Thus,
\[
    2\beta D\leq E\leq 2\gamma D.
\]
This verifies Theorem~3.1 for the discrete time scale.

The above computations show that the proposed kernel-type reverse Callebaut
inequality is consistent with both continuous and discrete time-scale settings.

\section{Conclusion}

In this paper, kernel-type fractional extensions of reverse Callebaut, Rogers--H\"older and Cauchy--Schwarz inequalities have been established on time scales. By introducing a nonnegative kernel into the diamond-$\alpha$ integral framework, the obtained results provide a unified approach for treating weighted and memory-dependent dynamic inequalities. The main results extend known reverse inequalities by incorporating kernel effects while preserving the structure of the classical estimates through reverse Young-type inequalities and Kantorovich-type bounds. Several particular cases have also been discussed, showing that the proposed inequalities reduce to the usual diamond-$\alpha$ dynamic inequalities when the kernel is chosen suitably. Moreover, the continuous and discrete validations demonstrate that the results are consistent with the time-scale framework and recover meaningful integral and summation forms.

The present work highlights the usefulness of kernel-type operators in extending classical reverse inequalities to broader dynamic settings. Since many models involving delay, heredity and memory can be described through integral kernels, the proposed inequalities may serve as useful tools in deriving a priori bounds and stability estimates for Volterra-type and delay dynamic equations on time scales.

As a possible future direction, one may investigate analogous reverse inequalities for fractional dynamic operators involving Caputo-type, Hilfer-type or conformable derivatives on time scales. Another interesting open problem is to apply the present kernel-type reverse inequalities to obtain explicit stability and boundedness criteria for nonlinear fractional dynamic equations with delay and impulsive effects.


\end{document}